\documentclass{amsart}

\usepackage[ruled,vlined, titlenotnumbered, algo2e]{algorithm2e}
\usepackage{amsthm}
\usepackage{amsmath,latexsym}
\usepackage{amssymb}
\usepackage{setspace}
\usepackage[english, activeacute]{babel}
\usepackage{float}
\usepackage{hyperref}
\usepackage{multicol}

\usepackage{color}
\usepackage[all]{xy}

\newcommand{\PF}{\mathrm{PF}}

\newcommand{\g}{\mathrm{g}}
\newcommand{\F}{\mathrm{F}}
\newcommand{\G}{\mathrm{G}}

\newcommand{\child}{\textnormal{\texttt{child}}}
\newcommand{\m}{\mathrm{m}}

\newcommand{\supp}{\mathrm{supp}}
\newcommand{\e}{\mathrm{e}}

\newcommand{\T}{\mathrm{T}}
\newcommand{\Ap}{\mathrm{Ap}}
\newcommand{\Z}{\mathbb{Z}}
\newcommand{\N}{\mathbb{N}}

\newcommand{\C}{\mathrm{C}}

\newcommand{\Sem}{\mathcal{S}}
\newcommand{\E}{\mathcal{E}}

\newcommand{\I}{\mathcal{I}}
\newcommand{\Kunz}{\mathcal{K}}
\newcommand{\dsum}{\displaystyle\sum}

\newtheorem{theo}{Theorem}
\newtheorem{ex}[theo]{Example}
\newtheorem{lem}[theo]{Lemma}
\newtheorem{cor}[theo]{Corollary}
\newtheorem{prop}[theo]{Proposition}

\newtheorem{rem}[theo]{Remark}
\usepackage[left=3cm, right=3cm, top=3cm, bottom=3cm]{geometry}
\begin{document}

\title[Numerical semigroups of a given genus]{The set of numerical semigroups of a given genus}

\author[V. Blanco and J.C. Rosales]{V. Blanco and J.C. Rosales\medskip\\
D\lowercase{epartamento de } \'A\lowercase{lgebra}, U\lowercase{niversidad de } G\lowercase{ranada}\medskip\\
\texttt{\lowercase{vblanco@ugr.es}} \; \texttt{\lowercase{jrosales@ugr.es}}}

\address{Departamento de \'Algebra, Universidad de Granada}



\begin{abstract}
In this paper we present a new approach to construct the set of numerical semigroups with a fixed genus. Our methodology is based on the construction of the set of numerical semigroups with fixed Frobenius number and genus. An equivalence relation is given over this set and a tree structure is defined for each equivalence class. We also provide a more efficient algorithm based in the translation of a numerical semigroup to its so-called Kunz-coordinates vector.
\end{abstract}

\maketitle

\section{Introduction}

A numerical semigroup is a subset $S$ of $\N$ (here $\N$ denotes the set of nonnegative integer numbers) closed under
addition, containing zero and such that $\N\backslash S$ is finite. The cardinal of the set $\N \backslash S$ is called the \textit{genus} of $S$ and we denoted it by $\g(S)$.

If $g$ is a positive integer, we denote by $\Sem(g)$ the set of numerical semigroups with genus $g$. The main goal of this paper is to provide an algorithm that allow to compute $\Sem(g)$. The problem of determining the elements in $\Sem(g)$ has been previously addressed by Blanco et. al \cite{counting}, Bras-Amor\'s \cite{bras08a,bras08b}, Bras-Amor\'os and Bulygin \cite{bras08b}, Elizalde \cite{elizalde}, Kaplan \cite{kaplan} and Zhao\cite{zhao}. However, the main goal of these papers was to count the number of elements in $\Sem(g)$ and get patterns, bounds or asymptotic behavior on these numbers. Actually, in \cite{bras08b} it is conjectured that the sequence of cardinals of $\Sem(g)$, for $g=1, 2, \ldots$, has a Fibonacci behavior.

Let $S$ be a numerical semigroup. The largest integer not belonging to $S$ is called the \textit{Frobenius number} (see \cite{alfonsin} for a detailed analysis of this number and related problems) of $S$ and we denote it by $\F(S)$. We represent by $\Sem(F,g)$ the set of numerical semigroups with Frobenius number $F$ and genus $g$. In Section \ref{sec:2} we prove that $\Sem(F,g)\neq \emptyset$ if and only if $g \leq F \leq 2g-1$ and an immediate consequence is that $\Sem(g) = \bigcup_{F=g}^{2g-1} \Sem(F,g)$. Thus, to solve the goal of this paper it is enough to design an algorithm to compute $\Sem(F,g)$ for any positive integers $F$ and $g$ such that $g \leq F \leq 2g-1$.

For any numerical semigroup $S$, the least positive integer belonging to $S$ is called the \textit{multiplicity} of $S$ and we denote it by $\m(S)$. We say that a numerical semigroup $S$ is \textit{elementary} if $\F(S) < 2\m(S)$. This family of semigroups will play an important role through this paper. We denote by $\E(F,g)$ the set of elementary numerical semigroups with Frobenius number $F$ and genus $g$.

In Section \ref{sec:2} we define an equivalence binary relation over $\Sem(F,g)$. The quotient set $\frac{\Sem(F,g)}{R} = \{[S]: S \in \Sem(F,g)\}$ provides us a partition of  $\Sem(F,g)$. We prove that each class contains an unique elementary numerical semigroup and then, $\Sem(F,g) = \bigcup_{S \in \E(F,g)} [S]$. Therefore, to enumerate the elements in $\Sem(F,g)$ we need:
\begin{enumerate}
 \item An algorithm to compute $\E(F,g)$.
\item A procedure that allows us to enumerate the elements in $[S]$ for each $S\in \E(F,g)$.
\end{enumerate}

Corollary \ref{cor:3} solves the first item. The second matter will be solved in Section \ref{sec:3} as follows. In Proposition \ref{prop:9} we see that the elements in $[S]$ can be organized in a rooted tree (with root $S$), and Proposition \ref{prop:11} informs us about how to construct the children of any vertex in the tree. Then, we get a recurent method to compute $[S]$.

In Section \ref{sec:4} we show how to apply the results in this paper to compute the set of irreducible numerical semigroups with a fixed Frobenius number. Note also that the results in this paper allows us to compute the set $\overline{\Sem}(F)$ of all the numerical semigroups with Frobenius number $F$ since $\overline{\Sem}(F) = \bigcup_{g=\lceil\frac{F+1}{2} \rceil}^F \Sem(F,g)$, where $\lceil\frac{F+1}{2} \rceil = \min\{x \in \Z: x \geq \frac{F+1}{2}\}$.

Based on the results in the previous sections, in Section \ref{sec:5} we provide a more efficient methodology to compute $\Sem(g)$ using the so-called Kunz-coordinates vector of a numerical semigroup. We see how to translate a numerical semigroup to a $0-1$ vector and that the computations over the semigroups to get the set $\Sem(g)$ can be easily done by manipulating these vectors. 

\section{Elementary Numerical Semigroups}
\label{sec:2}

The following result is well-known and appears in \cite{springer}.
\begin{lem}
 \label{lem:1}
If $S$ is a numerical semigroup, then $\g(S) \geq \frac{\F(S)+1}{2}$.
\end{lem}

For any finite set $X$, $\#X$ denotes the cardinal of $X$. If $a_1 < \cdots < a_n$ is a set of integer numbers, then, $\{a_1, \ldots, a_n, \rightarrow\} = \{a_1, \ldots, a_n\} \cup \{z \in \Z: z \geq a_n\}$.

\begin{prop}
 \label{prop:2}
Let $F$ and $g$ two positive integer. Then, the following conditions are equivalent:
\begin{enumerate}
 \item $g \leq F \leq 2g-1$.
\item $\E(F,g) \neq \emptyset$.
\item $\Sem(F,g) \neq \emptyset$.
\end{enumerate}
\end{prop}
\begin{proof} $ $
 \begin{itemize}
  \item[$(1)\Rightarrow(2)$] Note first that $g \leq F \leq 2g-1$ if and only if $\lceil\frac{F+1}{2} \rceil \leq g \leq F$. Observe also that $\#\{\lceil\frac{F+1}{2} \rceil, \ldots, F-1\}= \lceil\frac{F}{2} \rceil - 1$ and because $\lfloor\frac{F}{2}\rfloor + 1 \leq g$, then $F-g \leq \lceil\frac{F}{2} \rceil-1$.

Let $A$ be a subset of $\{\lceil\frac{F+1}{2} \rceil, \ldots, F-1\}$ with $\#A=F-g$. Since the addition of two elements in $A$ is greater or equal than $F+1$, we easily deduce that $\overline{S} = A \cup \{F+1, \rightarrow\} \cup \{0\}$ is a numerical semigroup. Furthermore, $\F(\overline{S}) = F$ and $\g(\overline{S}) = F - \#A = g$. Hence, $\overline{S} \in \Sem(F,g)$. Actually, $\m(\overline{S}) \geq \lceil\frac{F+1}{2} \rceil$ and consequently, $2\m(\overline{S}) \geq F+1> F$, being then $\overline{S} \in \E(F,g)$.
  \item[$(2)\Rightarrow(3)$] Trivial.
  \item[$(3)\Rightarrow(1)$] Let $S\in \Sem(F,g)$. It is clear that $g \leq F$ and by Lemma \ref{lem:1} we have that $g \geq \frac{F+1}{2}$. Thus, $g \leq F \leq 2g-1$.
 \end{itemize}
\end{proof}

The proof of $(1)\Rightarrow(2)$ in the proposition above is the key to construct $\E(F,g)$.

\begin{cor}
 \label{cor:3}
Let $F$ and $g$ two positive integers such that $g \leq F \leq 2g-1$, and let $\mathcal{A}_{F,g} = \{A: A \subseteq \{ \lceil\frac{F+1}{2} \rceil, \ldots, F-1\} \text{ and } \#A=F-g\}$. Then, $\E(F,g) = \{ A \cup \{F+1, \rightarrow\} \cup \{0\}: A \in \mathcal{A}_{F,g}\}$.
\end{cor}
\begin{proof}
Let $A \in \mathcal{A}_{F,g}$. Then, by the proof $(1)\Rightarrow (2)$ in Proposition \ref{prop:2}, we know that $A \cup \{F+1, \rightarrow\} \cup \{0\} \in \E(F,g)$. Let us see the other inclusion. Let $S \in \E(F,g)$. Then, $\m(S)>\frac{F}{2}$ and then $S$ is in the form $A \cup \{F+1, \rightarrow\} \cup \{0\}$ with $A \in \mathcal{A}_{F,g}$.
\end{proof}

As a immediate consequence we have the following result.

\begin{cor}
\label{cor:4}
Let $F$ and $g$ be positive integers such that $g \leq F\leq 2g-1$. Then
$$
\#\E(F,g) = {{\lceil\frac{F}{2} \rceil-1} \choose {F-g}}
$$
\end{cor}
We illustrate the above results with the following example.
\begin{ex}
\label{ex:5}
Let us construct the set of elementary numerical semigroups with Frobenius number $F=7$ and genus $g=5$. By Proposition \ref{prop:2} it is guaranteed that $\Sem(7,5)\neq \emptyset$ and by Corollary \ref{cor:4} $\#\E(7,5)={{4-1}\choose {7-5}}={3 \choose 2}=3$. Since $\{\lceil\frac{F+1}{2} \rceil, \ldots, F-1\} = \{4,5,6\}$, then, $\mathcal{A}_{7,5}$ described in Corollary \ref{cor:3} is $\mathcal{A}_{7,5}=\{A \subseteq \{4,5,6\}: \#A=2\} = \{\{4,5\}, \{4,6\}, \{5,6\}\}$ and by applying Corollary \ref{cor:3} we get that $\E(7,5) = \left\{\{0,4,5,8,\rightarrow\}, \{0,4,6,8, \rightarrow\}, \{0,5,6,8,\rightarrow\}\right\}$.
\end{ex}

\begin{prop}
\label{prop:6}
Let $F$ and $g$ be positive integers such that $g \leq F \leq 2g-1$. The correspondence $\theta: \Sem(F,g) \rightarrow \E(F,g)$ defined as $\theta(S)=\left( S \backslash \left\{x \in S\backslash\{0\}: x<\frac{F}{2}\right\}\right) \cup \{F-x: x \in S\backslash \{0\} \text{ and } x<\frac{F}{2}\}$ is a surjective application. Moreover, $\theta(S)=S$ if and only if $S\in \E(F,g)$.
\end{prop}
\begin{proof}
It is clear that $\theta(S)$ is in the form $A \cup \{F+1, \rightarrow\} \cup \{0\}$ where $A$ is a subset of $\{\lceil\frac{F+1}{2} \rceil, \ldots, F-1\}$ with cardinality $F-g$. Hence, by Corollary \ref{cor:3} we have that $\theta(S)\in \E(F,g)$ and then $\theta$ is a well defined application. It is also clear that $\theta(S)=S$ if and only if $S$ does not contain any positive integer smaller that $\frac{F}{2}$, which is equivalent to $S \in \E(F,g)$. Finally, as a consequence of the last statement, we have that $\theta$ is surjective.
\end{proof}
We define over the set $\Sem(F,g)$ the following equivalence relation:
$$
SRS' \text{ if and only if } \theta(S)=\theta(S').
$$
For a given $S \in \Sem(F,g)$ we denote by $[S]=\{S'\in \Sem(F,g): SRS'\}$. Then, the quotient set $\frac{\Sem(F,g)}{R} = \{[S]: S \in \Sem(F,g)\}$ is a partition of $\Sem(F,g)$. As an immediate consequence of Proposition \ref{prop:6} we get the following result.
\begin{cor}
\label{cor:7}
Let $F$ and $g$ be positive integers such that $g \leq F \leq 2g-1$. Then $\frac{\Sem(F,g)}{R} = \{[S]: S \in \E(F,g)\}$. Furthermore, if $S, S' \in \E(F,g)$ and $S\neq S'$, then $[S] \cap [S']=\emptyset$.
\end{cor}

Let $g$ be a positive integer and $X=\{F \in \N\backslash\{0\}: g \leq F \leq 2g-1\}$. Note that by Proposition \ref{prop:2} we know that $\Sem(g) = \bigcup_{F\in X} \Sem(F,g)$ and by Corollary \ref{cor:7} that $\Sem(F,g) = \bigcup_{S\in \E(F,g)} [S]$. Corollary \ref{cor:3} gives us an algorithmic procedure to compute $\E(F,g)$. Hence, to compute $\Sem(F,g)$ it is enough to provide an algorithm to compute the elements in $[S]$ for any $S \in \E(F,g)$. The next section will give us the answer to this question.

\section{The tree associated to an elementary numerical semigroup}
\label{sec:3}
In this section we define a tree structure for $[S]$, for any $S\in \E(F,g)$ and any positive integers $F$ and $g$. The use of this tree leads us to describe a method for enumerating the elements of $[S]$.

A (directed) graph $G$ is a pair $(V,E)$ where $V$ is a nonempty set whose elements are called vertices and $E$ is a subset of $\{(v,w) \in V\times V: v\neq w\}$. The elements in $E$ are called edges. A path connecting two vertices $v$ and $w$ of $G$ is a finite sequence of different edges in the form $(v_0, v_1), (v_1, v_2), \ldots, (v_{n-1}, v_n)$ with $v_0=v$ and $v_n=w$.

A graph $G$ is said a tree if there exists a vertex $r$ (called the root of $G$) such that for any other vertex of $G$, $v$, there is an unique path connecting $v$ and $r$. If $(v,w)$ is an edge of the tree we say that $v$ is a child of $w$.

Let $S \in \E(F,g)$. Our first goal in this section is to construct a tree whose set of vertices is $[S]$ and its root is $S$.

\begin{lem}
\label{lem:8}
Let $S$ be a numerical semigroup such that $\F(S)>2\m(S)$. Then, $\overline{S}=\left(S\backslash\{\m(S)\}\right) \cup \{\F(S)-\m(S)\}$ is a numerical semigroup with $\F(\overline{S})=\F(S)$, $\g(\overline{S})=\g(S)$ and $\m(\overline{S})>\m(S)$. Moreover, $\overline{S} \in [S]$.
\end{lem}

\begin{proof}
$\overline{S}$ is a semigroup as a immediate consequence of the following items:
\begin{itemize}
\item The addition of two elements in $S\backslash\{\m(S)\}$ belongs to $S\backslash\{\m(S)\}$.
\item $2(\F(S)-\m(S)) = \F(S) + (\F(S)-\m(S)) > \F(S)$ and then, $2(\F(S)-\m(S))\in S\backslash\{\m(S)\}$.
\item If $s\in S\backslash\{\m(S), 0\}$, then $s>\m(S)$. Hence, $\F(S)-\m(S)+s >\F(S)$ and $\F(S)-\m(S)+s \in S\backslash\{\m(S)\}$.
\end{itemize}

By definition of $\overline{S}$, we deduce that $\g(\overline{S})=\g(S)$ and $\F(\overline{S})=\F(S)$. Moreover, since $\F(S)-\m(S)>\m(S)$ we have that $\m(\overline{S})>\m(S)$.

Finally, it is clear that $\theta(\overline{S})=\theta(S)$ and then $SR\overline{S}$.
\end{proof}

The above lemma allows us to give the following definition. Let $S \in \Sem(F,g)$. We define the graph $\G([S])$ as follows: The set of vertices is $[S]$ and $(P,Q) \in [S] \times [S]$ is an edge if $F>2\m(P)$ and $Q=\left(P\backslash \{\m(P)\}\right) \cup \{F-\m(P)\}$.

\begin{prop}
\label{prop:9}
Let $S \in \E(F,g)$. Then, $\G([S])$ is a tree with root $S$.
\end{prop}
\begin{proof}
Let $P \in [S]$. By applying iteratively Lemma \ref{lem:8} we have the following sequence of elements in $[S]$:
\begin{itemize}
\item[] $P_0=P$,
\item[] $P_{n+1} = \left\{\begin{array}{cl}
\left(P_n\backslash \{\m(P_n)\}\right) \cup \{F-\m(P_n)\} & \mbox{ if $F>2\m(P_n)$,}\\
P_n & \mbox{ otherwise}
\end{array}\right.$
\end{itemize}
It is clear that $k = \min \{\ell \in \N: P_\ell \in \E(F,g)\}$ exists. By applying Corollary \ref{cor:7} we have that $P_k=S$. Therefore, $(P_0, P_1), (P_1, P_2), \ldots, (P_{k-1}, P_k)$ is the unique path connecting $P$ and $S$. Consequently, $\G([S])$ is a tree with root $S$.
\end{proof}

Observe that by Corollary \ref{cor:7}, Proposition \ref{prop:9} and that $\Sem(F,g) = \bigcup_{S \in \E(F,g)} [S]$, we have that, with the above graph structure, $\Sem(F,g)$ is a forest, i.e., a disjoin union of trees. Hence, $\Sem(g)$ is also a forest with this construction.

Note that to construct $[S]$ it is enough, by the proposition above, to provide an algorithm to determine the children of any $Q\in [S]$ in $\G([S])$. This is the next goal in this section. First we recall some definitions.

Let $A$ be a nonempty subset in $\N$. We denote by $\langle A \rangle$ the submonoid of $(\N, +)$ generated by $A$, that is, $\langle A \rangle = \{\lambda_1a_1 + \cdots + \lambda_na_n: n\in \N \backslash \{0\}, a_1, \ldots, a_n \in A, \text{ and } \lambda_1, \ldots, \lambda_n \in \N\}$. It is well-known (see for instance \cite{springer}) that $\langle A \rangle$ is a numerical semigroup if and only if $\gcd(A)=1$ (here $\gcd$ stands for the greatest common divisor). If $S$ is a numerical semigroup and $S=\langle A \rangle$, then we say that $A$ is a system of generators of $S$. If there not exists any other proper subset of $A$ generating $S$, we say that $A$ is a \emph{minimal system of generators} of $S$. Every numerical semigroup has an unique minimal system of generators and such a system is finite (see \cite{springer}). If $S$ is a numerical semigroup, the elements in a minimal system of generators of $S$ are called the \textit{minimal generators} of $S$.

 The following result is well-known (see for instance \cite{springer}).

 \begin{lem}
 \label{lem:10}
 Let $S$ be a numerical semigroup.
 \begin{enumerate}
 \item The minimal generators of $S$ are the elements in $S\backslash\{0\}$ that cannot be expressed as a sum of two elements in $S\backslash\{0\}$.
 \item Let $x\in S$. Then, $x$ is a minimal generator of $S$ if and only if $S\backslash \{x\}$ is a numerical semigroup.
 \end{enumerate}
 \end{lem}

 Let $S$ be a numerical semigroup. Following the notation in \cite{jpaa}, an element $x \in \Z\backslash S$ is a \textit{pseudo-Frobenius number} of $S$ if $x+s\in S$ for all $s\in S\backslash\{0\}$. We denote by $\PF(S)$ the set of pseudo-Frobenius numbers of $S$. Note that $\#\PF(S)$ is an important invariant of $S$ called the \textit{type} of $S$ (see for instance \cite{barucci}).

 \begin{prop}
 \label{prop:11}
 Let $S \in \E(F,g)$ and $Q\in [S]$. Then, $P$ is a child of $Q$ in $\G([S])$ if and only if $P=\left(Q\backslash\{x\}\right) \cup \{F-x\}$, where $x$ is a minimal generator of $Q$ verifying the following conditions:
 \begin{enumerate}
 \item\label{11:1} $\frac{F}{2} <x<F$,
 \item\label{11:2} $F-x < \m(Q)$,
 \item\label{11:3} $2F\neq 3x$,
 \item\label{11:4} $2(F-x)$ is a minimal generator of $Q$,
 \item\label{11:5} $F-x\in \PF(Q\backslash\{x\})$.
 \end{enumerate}
 \end{prop}

 \begin{proof}
 {\rm (Sufficiency)} Let us see first that $P=\left(Q \backslash \{x\}\right) \cup \{F-x\}$ is a numerical semigroup. By Lemma \ref{lem:10} we know that the addition of two elements in $Q\backslash\{x\}$ is an element in $Q\backslash\{x\}$. By condition \eqref{11:4} we have that $2(F-x)\in Q$ and by \eqref{11:3} that $2(F-x)\neq x$. Hence, $2(F-x)\in Q\backslash\{x\}$. Finally, by \eqref{11:5} we know that $F-x+s \in Q\backslash\{x\}$ for all $s \in Q\backslash\{x,0\}$.

 From \eqref{11:1} we deduce that $F-x<\frac{F}{2}$, so $\theta(P)=\theta(Q)$. Thus, $P \in [Q]=[S]$.

 By $\eqref{11:2}$ we have that $F-x < \m(Q)$ and then $F-x = \m(P)$. Therefore, $Q=\left(P\backslash\{\m(P)\}\right) \cup \{F-\m(P)\}$. Furthermore, since $x>\frac{F}{2}$ we have that $F>2(F-x)=2\m(P)$. Then, $P$ is a child of $Q$.

  {\rm (Necessity)} Let $P$ be a child of $Q$. Then $Q=\left(P\backslash\{\m(P)\}\right) \cup \{F-\m(P)\}$ and $F>2\m(P)$. Hence, $P=\left(Q\backslash\{F-\m(P)\}\right) \cup \{F-(F-\m(P))\}$. To conclude the proof we need to see that $F-\m(P)$ is a minimal generator of $Q$ verifying the conditions \eqref{11:1}-\eqref{11:5}.

  Since $F-\m(P) \not\in P$ and $Q=\left(P\backslash\{\m(P)\}\right) \cup \{F-\m(P)\}$ then by Lemma \ref{lem:10} we deduce that $F-\m(P)$ is a minimal generator of $Q$. Let us see the conditions:
  \begin{enumerate}
  \item Since $0<\m(P)<\frac{F}{2}$, then $\frac{F}{2}<F-\m(P)<F$.
  \item\label{11:2a} $F-(F-\m(P)) = \m(P)<\m(Q)$ since $Q=\left(P\backslash\{\m(P)\}\right) \cup \{F-\m(P)\}$ and $F-\m(P)>\m(P)$.
  \item Suppose that $2F= 3(F-\m(P))$, then $F=3\m(P) \in P$ contradicting that $F$ is the Frobenius number of $P$.
  \item $2(F-(F-\m(P))= 2\m(P)\in Q$. Furthermore, by \eqref{11:2a} we know that $\m(P)<\m(Q)$. By applying Lemma \ref{lem:10} we deduce that $2\m(P)$ is a minimal generator of $Q$.
  \item $F-(F-\m(P)) = \m(P)\not\in Q$. Since $P\backslash \{\m(P)\} = Q\backslash\{F-\m(P)\}$ and clearly $\m(P)\in \PF(P\backslash\{\m(P)\}$, we have that $\m(P) \in \PF(Q\backslash\{F-\m(P)\}$.
  \end{enumerate}
 \end{proof}

 To conclude this section we illustrate the above results in the following example.
 \begin{ex}
 \label{ex:12}
 Let us construct the set of numerical semigroups with Frobenius number $7$ and genus $5$. By Example \ref{ex:5} we know that $\E(7,5)=\{\langle4,5,11\rangle, \langle 4,6,9,11\rangle, \langle 5,6,8,9\rangle\}$. By Corollary \ref{cor:7} we have that $\Sem(7,5)=[\langle4,5,11\rangle] \cup [\langle 4,6,9,11\rangle] \cup [\langle 5,6,8,9\rangle]$.

 Now, by applying Proposition \ref{prop:11} we deduce that $\langle 4,5,11\rangle$ and $\langle 5,6,8,9\rangle$ have no children. Hence, $[\langle 4,5,11\rangle] = \{\langle 4,5,11\rangle\}$ and $[\langle 5,6,8,9\rangle] = \{\langle 5,6,8,9\rangle\}$. $\langle 4,6,9, 11\rangle$ has only one children $\left( \langle 4,6,9,11\rangle \backslash \{4\}\right) \cup \{3\} = \langle 3,8,10\rangle$. By applying again Proposition \ref{prop:11} we get that $\langle 3,8,10\rangle$ has no children, so $[\langle 4,6,9,11\rangle] = \{\langle 4,6,9,11\rangle, \langle 3,8,10\rangle\}$.

 Then, $\Sem(7,5) = \{\langle4,5,11\rangle, \langle 4,6,9,11\rangle, \langle 3,8,10\rangle, \langle 5,6,8,9\rangle\}$.
In Figure \ref{fig:ex1} we show the forest structure of $\Sem(7,5)$.

 \begin{figure}[h]
\framebox{
$$
\xymatrix{\langle 4,5,11\rangle & & \langle 5,6,8,9\rangle\\
& \langle 4,6,9,11 \rangle &\\
& \langle 3,8, 10 \rangle \ar[u]&}
$$}
\caption{Forest structure of $\Kunz(7,5)$.\label{fig:ex1}}
 \end{figure}

 \end{ex}
\section{Irreducible numerical semigroups}
\label{sec:4}
A numerical semigroup is said \textit{irreducible} if it cannot be expressed as an intersection of two numerical semigroups containing it properly. This notion was introduced in \cite{pacific} where it is proven that a numerical semigroup $S$ is irreducible if and only if $S$ is maximal (with respect to the inclusion ordering) in the set of numerical semigroups with Frobenius number $\F(S)$. From \cite{barucci} and \cite{froberg} it is deduced that the family of irreducible numerical semigroups is the union of two families of numerical semigroups with special importance in this theory: symmetric and pseudo-symmetric numerical semigroups. In the literature there are many characterization of this type of semigroups. The following result is well-known and appears in \cite{springer}.

\begin{lem}
\label{lem:13}
Let $S$ be a numerical semigroup. Then:
\begin{enumerate}
\item $S$ is symmetric if and only if $\F(S)$ is odd and $\g(S)=\frac{\F(S)+1}{2}$.
\item $S$ is pseudo-symmetric if and only if $\F(S)$ is even and $\g(S)=\frac{\F(S)+2}{2}$.
\item $S$ is irreducible if and only if $\g(S)=\left\lceil\frac{\F(S)+1}{2}\right\rceil$.
\end{enumerate}

\end{lem}

We denote by $\I(F)$ the set of irreducible numerical semigroups with Frobenius number $F$. Note that the results in this paper give us an algorithm to compute $\I(F)$ since as a consequence of Lemma \ref{lem:13} we have that $\I(F) = \Sem(F, \left\lceil\frac{F+1}{2}\right\rceil)$. Observe also that $F - \left\lceil\frac{F+1}{2}\right\rceil$ coincides with the cardinal of $\{\left\lceil\frac{F+1}{2}\right\rceil, \ldots, F-1\}$. Hence, by applying Corollary \ref{cor:3} we have that $\E(F, \lceil\frac{F+1}{2}\rceil)$ has only one element $\T(F)=\{\left\lceil\frac{F+1}{2}\right\rceil, \ldots, F-1\} \cup \{F+1, \rightarrow\} \cup \{0\}$. Therefore, $\I(F)=[\T(F)]$ and by applying Proposition \ref{prop:9} we have that $\I(F)$ is a tree with root $\T(F)$. In this way we have reproved some of the results in \cite{arbol}. However, the algorithm in \cite{arbol} to compute $\I(F)$ is more efficient to the one in this section.

In \cite{arbol}, a tree structure is defined over $\I(F)$ where the root is the numerical semigroup $\C(F)=
\{0, \lceil\frac{F+1}{2}\rceil, \rightarrow\} \backslash \{F\}$ coinciding with $\T(F)$. The reader can easily check that, in fact, both trees are the same. Thus, the construction given here extends the construction in \cite{arbol}. However, the procedure given in \cite{arbol} to construct the tree takes advantage of the irreducibility of the semigroups, speeding-up the algorithm.

Finally, based on the results in \cite{arbol}, the authors give in \cite{semilattice} an algorithm to compute $\overline{\Sem}(F)$, the set of numerical semigroups with Frobenius number $F$. As mentioned at the end of the introductory section, this set can also be computed using the approach in this paper. However, the algorithm in \cite{semilattice} is more efficient than the adaptation of the techniques presented here.

The reader may think, after the above mentioned comments, that the algorithm in \cite{semilattice} can be adapted to compute $\Sem(F,g)$. Actually, it is possible by computing first $\overline{\Sem}(F)$ by the algorithm in \cite{semilattice} and then, to construct $\Sem(F,g)$ it is enough to filter the elements in $\overline{\Sem}(F)$ to those with genus $g$. Clearly, this method is less efficient that the approach presented in this paper since we only enumerate the elements that we want to get and no filters are applied. Note that also we may think to construct $\overline{\Sem}(F)$ with our methodology since $\overline{\Sem}(F) = \bigcup_{\lceil\frac{F+1}{2}\rceil\leq g \leq F} \Sem(F,g)$.

\section{A Kunz-coordinates vector approach to compute $\Sem(g)$}
\label{sec:5}

In this section we translate the main results in the sections above to the so-called Kunz-coordinates vectors in order to compute, more efficiently, the set $\Sem(g)$ for any positive integer $g$. For the sake of completeness we construct the Kunz-coordinates vector of the elements in $\Sem(g)$ from its Ap\'ery set.

 Let $S$ be a numerical semigroup and $n \in S\backslash \{0\}$. The \emph{Ap\'ery set} of $S$ with respect to $n$ is the set $\Ap(S,n) = \{s
\in S :  s - n \not\in S\}$. This set was introduced by Ap\'ery in \cite{apery}.

The following characterization  of the Ap\'ery set that appears in \cite{springer} is crucial for our development.
 \begin{lem}
 \label{lem:14}
Let $S$ be a numerical semigroup and $n \in S\backslash \{0\}$. Then $\Ap(S,n) = \{0 = w_0, w_1, \ldots,  w_{n -
1}\}$, where $w_i$ is the least
element in $S$ congruent with $i$ modulo $n$, for $i=1, \ldots, n-1$.
 \end{lem}

Moreover, the set $\Ap(S,n)$ completely determines $S$, since $S =
\langle \Ap(S,n) \cup \{n\} \rangle$ (see \cite{london}), and then, we can identify $S$ with its Ap\'ery set with respect to $n$. The set $\Ap(S,n)$ contains, in general, more information than an arbitrary system of
generators of $S$. For instance, Selmer in \cite{selmer77} gives the formulas,
$\g(S)=\frac{1}{n}\left(\sum_{w \in \Ap(S, n)} w\right) -
\frac{n-1}{2}$ and $\F(S) = \max(\Ap(S,n)) - n$. One can also test if a nonnegative integer $s$ belongs to $S$ by checking if $w_{s\pmod m} \leq s$.

Let $g$ be a positive integer and $S\in \Sem(g)$. By Lemma \ref{lem:1}, $\F(S) \leq 2g-1$, and then $2g>\F(S)$. Thus, $2g\in S\backslash\{0\}$. Then, with this settings, we refer to the Ap\'ery set as the Ap\'ery set with respect to $2g$.

We consider an useful modification of the Ap\'ery set that we call the \emph{Kunz-coordinates vector} as in \cite{bp2011}. Let $S$ be a numerical semigroup and $n \in S\backslash\{0\}$. The \emph{Kunz-coordinates vector} of $S$ with respect to $n$ is the vector $\Kunz(S,n) = x \in \N^{n-1}$ with components $x_i = \frac{w_i-i}{n}$ for $i=1, \ldots, n-1$, where $\Ap(S, n)=\{w_0=0, w_1, \ldots, w_{n-1}\}$, with $w_i$ congruent with $i$ modulo $n$ (Lemma \ref{lem:13}). If $x \in \N^{n-1}$ is a Kunz-coordinates vector, we denote by $S_x$ the unique numerical semigroup such that $\Kunz(S_x, n) = x$.
The Kunz-coordinates vectors were introduced in \cite{kunz} and have been previously analyzed when $n=\m(S)$ in \cite{london}. As mentioned for the Ap\'ery set, through this paper we refer to the Kunz-coordinates vector of a numerical semigroup $S$ as the Kunz-coordinates vector of $S$ with respect to $2\g(S)$, that is, we denote $\Kunz(S)= \Kunz(S, 2\g(S))$. We denote by $\Kunz(g) = \{\Kunz(S): S \in \Sem(g)\}$, the set of Kunz-coordinates vectors of the numerical semigroups with genus $g$. Also, if $F$ is a positive integer $\Kunz(F,g) = \{\Kunz(S): S \in \Sem(F,g)\}$.

Note that if $x\in \Kunz(g)$, $x_i$ is one if $i \not\in S_x$ and zero otherwise. Then, the following lemmas has a trivial proof. There, for an integer vector $x \in \N^n$ we denote by $\supp(x) = \{i \in \{1, \ldots, n\}: x_i \neq 0\} \subseteq \{1, \ldots, n\}$ the \emph{support} of $x$.
\begin{lem}
\label{lem:14}
Let $S$ be a numerical semigroup with genus $g$ and $x=\Kunz(S) \in \N^{2g-1}$. Then:
\begin{enumerate}
\item\label{14:1} $ x \in \{0,1\}^{2g-1}$,
\item\label{14:2} $\N\backslash S = \supp(x)$,
\item\label{14:3} $\F(S) = \max\{i \in \{1, \ldots, 2g-1\}: x_i=1\}$.
\item\label{14:4} $\m(S) = \min\{i \in \{1, \ldots, 2g-1\}: x_i=0\}$.
\item\label{14:5} The set of minimal generators of $S$ is the set $\{i \in \{1, \ldots, 2g-1\}: x_i=0 \text{ and } x_j + x_{i-j} \geq 1 \text{ for all } 1 \leq j <i\}$.
\item\label{14:6} $\PF(S) = \{i \in \{1, \ldots, 2g-1\}: x_i =1, x_j \geq x_{i+j} \quad \text{, for all } j=1, \ldots, 2g-1-i\}$.
\end{enumerate}
\end{lem}

\begin{lem}
\label{lem:15}
Let $F$ and $g$ be positive integers such that $g\leq F\leq 2g-1$. Then:
\begin{enumerate}
\item $\Kunz(F,g) = \{x \in \{0,1\}^{2g-1}: x_F=1, x_{i}=0 \text{ for all } F< i <2g \text{,} x_i+x_j-x_{i+j}\geq 0 \text{ for all } 1\leq i \leq j <2g \text{ with } i+j \leq F\text{, and } \sum_{i=1}^{F} x_i=g\}$.
    \item $\Kunz(g) = \{x \in \{0,1\}^{2g-1}: x_i+x_j-x_{i+j}\geq 0 \text{ for all } 1\leq i \leq j < 2g \text{ with } i+j \leq 2g-1\text{, and } \sum_{i=1}^{2g-1} x_i=g\}$.
\end{enumerate}
\end{lem}

We denote by $\PF^K(x) = \PF(S_x)$ for any $x \in \Kunz(g)$ and any positive integer $g$.

From the above result, $\varphi: \Sem(F,g) \longrightarrow \Kunz(F,g)$, defined as $\varphi(S) = \Kunz(S)$  is a bijective application. Actually, $\varphi^{-1}(x) = \langle 2g, 2gx_1+1, \ldots, 2gx_{2g-1}+2g-1\rangle$. Then, computing the set $\Sem(g)$ is equivalent to compute $\Kunz(g)$.

As in the sections above, to construct $\Sem(g)$ we construct first $\Sem(F,g)$ for $g \leq F \leq 2g-1$. And to construct $\Sem(F,g)$ we defined an equivalence relation over $\Sem(F,g)$ that gave us a partition of this set by its equivalence classes. For the sake of translating the equivalence relation to Kunz-coordinates vectors, we first analyze the set of Kunz-coordinates vectors of the elements in $\E(F,g)$ for any positive integers $F$ and $g$ with $g \leq F \leq 2g-1$.

We denote by $\E^K(F,g) = \{\Kunz(S): S \in \E(F,g)\}$ and we call it the set of elementary Kunz-coordinates vectors with Frobenius number $F$ and genus $g$. By Lemma \ref{lem:14}, we have that $\E^K(F,g) = \{x \in \Kunz(F,g): \max\{i: x_i=1\} < 2\min\{i: x_i=0\}$. And from Corollary \ref{cor:3} we easily deduce the following result where for any set $A\subseteq \{1, \ldots, m\}$ and $n\geq m$ we denote by $\chi^n_A=(x_i)$ the vector in $\N^n$ with coordinates:
$$
x_i = \left\{\begin{array}{rl}
0 & \mbox{if $i \in A$ or $i>m+1$,}\\
1 & \mbox{otherwise}
\end{array}\right. \quad \text{ for } i=1, \ldots, n.
$$
\begin{lem}
\label{lem:16}
Let $F$ and $g$ positive integers. Then:
$$
\E^K(F,g) = \{\chi^{2g-1}_A : A \subseteq \mathcal{A}_{F,g}\}
$$
where $\mathcal{A}_{F,g}= \{A \subseteq \{\lceil\frac{F+1}{2}\rceil, \ldots, F-1\}: \#A=F-g\}$.
\end{lem}

Next, to define the analogous equivalence relation of the one over $\Sem(F,g)$ but for $\Kunz(F,g)$ we give the following immediate result. For $i \in \{1, \ldots, 2g-1\}$, we denote by $\e_i$ the $2g-1$-tuple having a $1$ as its $i$th entry and zeros otherwise.
\begin{lem}
\label{lem:17}
Let $S$ be a numerical semigroup with genus $g$, $x=\Kunz(S)$ and $i \in \{1, \ldots, 2g-1\}$. Then:
\begin{enumerate}
\item If $i$ is a minimal generator of $S$, then $\Kunz(S \backslash\{i\}) = x + \e_i$.
\item If $i \not\in S$ such that $S \cup \{i\}$ is a numerical semigroup,  then $\Kunz(S \cup \{i\}) = x - \e_i$.
\end{enumerate}
\end{lem}

We define $\theta^K: \Kunz(F,g) \longrightarrow \E^K(F,g)$ as $\theta^K(x) = x + \dsum_{i<\frac{F}{2}: x_i=0} \e_i - \dsum_{i<\frac{F}{2}: x_i=0} \e_{F-i}$. Note that by Lemma \ref{lem:17}, $\Kunz(\theta(S_x))=\theta^K(x)$. By Proposition \ref{prop:6} we have that $\theta^K$ is a surjective application and such that $\theta^K(x)=x$ if and only if $x \in \E^K(F,g)$.

We are now ready to define the announce equivalence relation over $\Kunz(F,g)$. Let $x, y \in \Kunz(F,g)$, $xR^K y$ if $\theta^K(x) = \theta^K(y)$. As an immediate consequence of Corollary \ref{cor:7} we have the following result.

\begin{cor}
\label{cor:18}
Let $F$ and $g$ be positive integers with $g \leq F \leq 2g-1$. Then, $\frac{\Kunz(F,g)}{R^K} = \{[x]: x \in \E^K(F,g)\}$. Furthermore, if $x, y \in \E^K(F,g)$ and $x\neq y$, then $[x] \cap [y] =\emptyset$.
\end{cor}

We can also define an analogous tree structure over $[x]$ for each $x\in \E^K(F,g)$. Let $\G^K([x])$ be the graph with vertices the elements in $[x]$ and $(y,z)$ is an edge if $z=y+\e_\m-\e_{F-\m}$ where $\m=\min\{i \in \{1, \ldots, 2g-1\}: y_i=0\}$. By Proposition \ref{prop:9}, $\G^K([x])$ is a tree with root $x$. Furthermore, we can easily detect the children of any vertex in $\G^K([x])$ as in Proposition \ref{prop:11}.

 \begin{prop}
 \label{prop:19}
 Let $x \in \E^K(F,g)$ and $z\in [x]$. Then, $y$ is a child of $z$ in $\G^K([x])$ if and only if $y=z+\e_i-\e_{F-i}$, where $i$ verifies:
 \begin{enumerate}
 \item\label{19:1} $z_{i}=0$,
 \item\label{19:2} $z_j + z_{i-j} \geq 1$, for all $j <i$,
 \item\label{19:3} $\frac{F}{2} <i<F$,
 \item\label{19:4} $F-i < \min\{i \in \{1, \ldots, 2g-1\}: z_i=0\}$,
 \item\label{19:5} $2F\neq 3i$,
  \item\label{19:6} $z_{2(F-i)}=0$,
 \item\label{19:7} $z_j + z_{2(F-i)-j} \geq 1$, for all $j <2(F-i)$,
 \item\label{19:8} $F-i\in \PF^K(z+\e_i)$.
 \end{enumerate}
 \end{prop}
For any $z  \in [x]$ for $x \in \E^K(F,g)$, we denote by $\Gamma(z)$ the set of elements in $\{1, \ldots, 2g-1\}$ verifying the conditions \eqref{19:1}-\eqref{19:8} in the above proposition.

A pseudocode for the algorithm proposed to enumerate $\Kunz(g)$ is described in Algorithm \ref{alg:1}.

\begin{algorithm2e}[H]
\caption{Computation of $\Kunz(g)$.\label{alg:1}}

\SetKwInOut{Input}{input}
\SetKwInOut{Output}{output}
\Input{A positive integer $g$.}

\For{$F \in \{g, \ldots, 2g-1\}$}{
 Compute  $\E^K(F,g)$.

 \For{$x \in \E^K(F,g)$}{
 Compute the tree $\G^K([x])$ and set $\child=\{x\}$.

 \While{$\child\neq\emptyset$}{

 \For{$y \in \child$}{
 \begin{enumerate}
 \item $\child = \child \backslash \{y\}$,
 \item  Compute $\Gamma(y)$.
 \item $\child=\child \cup \bigcup_{i \in \Gamma(y)} \{y + \e_i-\e_{F-i}\}$.
 \end{enumerate}
 Set $S= S \cup \child$.
 }}}}

\Output{$\Sem(g)$.}
\end{algorithm2e}
Observe that if $x \in \Kunz(F,g)$, by Lemma \ref{lem:16}, $x\in \{0,1\}^{2g-1}$ and then, to construct its children we only need to interchange two coordinates, $i$ and $F-i$, where $i$ verifies the conditions of Proposition \ref{prop:19}.
 \begin{ex}
 \label{ex:20}
 Let us construct the set $\Kunz(5)$. First, we compute $\E^K(F,5)$ for $5 \leq F \leq 9$.
 \begin{itemize}
 \item For $F=5$, Since $5-5=0$, $\E^K(5,5)=\{(1,1,1,1,1,0,0,0,0)\}$.
 \item For $F=6$, $\mathcal{A}_{6,5}=\{4,5\}$ and $6-5=1$, so $\E^K(6,5)=\{(1,1,1,0,1,1,0,0,0)$, $(1,1,1,1,0,1,0,0,0)\}$.
 \item For $F=7$, $\mathcal{A}_{7,5}=\{4,5,6\}$ and $7-5=2$, so $\E^K(7,5)=\{(1,1,1,0,0,1,1,0,0)$, $(1,1,1,0,1,0,1,0,0)$,  $(1,1,1,1,0,0,1,0,0)\}$.
 \item For $F=8$, $\mathcal{A}_{8,5}=\{5,6,7\}$ and $8-5=3$, so $\E^K(8,5)=\{(1,1,1,1,0,0,0,1,0)\}$.
  \item And for $F=9$, $\mathcal{A}_{9,5}=\{5,6,7,8\}$ and $9-5=4$, so $\E^K(8,5)=\{(1,1,1,1,0,0,0,0,1)\}$.
 \end{itemize}
 Let us now compute the children of the elements in $\E^K(F,5)$ for $5\leq F \leq 9$. For $x=(1,1,1,1,1,0,0,0,0)$, it is clear that it has no children since there are no indices in $\{\lceil\frac{5}{2}\rceil, \ldots, 5-1\}$ with $x_i=0$.\\ $(1,1,1,0,1,1,0,0,0)$, $(1,1,1,1,0,1,0,0,0)$ and $(1,1,1,0,0,1,1,0,0)$ have also no children. For $x=(1,1,1,0,1,0,1,0,0)$, we get that there is only one index, $i=4$, verifying the conditions of Proposition \ref{prop:19}, hence $y=x+\e_4-\e_3 = (1,1,0,1,1,0,1,0,0)$ is the unique child of $x$. $y$ has no children.

 Analogously, we can compute the tree associated to the class of each element in $\E^K(F,5)$ for  $5\leq F \leq 9$.In Figure \ref{fig:1} we show the forest $\{\G^K([x]): x \in \E^K(F,5), 5\leq F \leq 9\}$.

 Note that we can easily extract the numerical semigroups with genus $5$ from their Kunz-coordinates, $\Kunz(5)$. For instance, by applying that $S_x= \langle 10, 10x_1+1, \ldots, 10x_{9}+9 \rangle$, for any $x \in \Kunz(5)$, we get that $\Sem(5)=\{\langle  6, 7, 8, 9, 10, 11 \rangle, \langle  5, 7, 8, 9, 11 \rangle,$ $\langle  5, 6, 8, 9 \rangle,$ $\langle  5, 6, 7, 9 \rangle,$   $\langle  5, 6, 7, 8 \rangle$, $\langle  4, 6, 7 \rangle$, $\langle  4, 7, 9, 10 \rangle$, $\langle  4, 6, 9, 11 \rangle$,
  $\langle  4, 5, 11 \rangle$, $\langle  3, 8, 10 \rangle$, $\langle  3, 7, 11 \rangle$, $\langle  2, 11 \rangle \}$.

 \begin{figure}[h]
\framebox{
 {\scriptsize$
 \begin{array}{cccc}
\xymatrix{(1,1,1,1,1,0,0,0,0)} & \xymatrix{(1,1,1,0,1,1,0,0,0)} & \xymatrix{(1,1,1,1,0,1,0,0,0)} & \xymatrix{(1,1,1,0,0,1,1,0,0)}\\
\xymatrix{\\(1,1,1,0,1,0,1,0,0)\\(1,1,0,1,1,0,1,0,0)\ar[u]} & \xymatrix{\\(1,1,1,1,0,0,1,0,0)} & \xymatrix{\\(1,1,1,1,0,0,0,1,0)\\(1,1,0,1,1,0,0,1,0)} &
\xymatrix{(1,1,1,1,0,0,0,0,1)\\(1,1,1,0,1,0,0,0,1)\ar[u]\\ (1,0,1,0,1,0,1,0,1)\ar[u]}
 \end{array}
$}}
\caption{Forest structure of $\Kunz(5)$.\label{fig:1}}
 \end{figure}

 \end{ex}

\begin{rem}
The approach in this paper also allows us to count the numerical semigroups with a more efficient method that those applied in \cite{counting,bras08a,bras08b,elizalde,zhao} since when applying our methodology we do not storage in memory the whole set of numerical semigroups in $\Sem(g)$ or the Kunz-coordinates vectors in $\Kunz(g)$ that are computed in the overall procedure. Observe that to compute the number of elements in $\Sem(g)$ we compute $\Sem(F,g)$ for $F=g, \ldots, 2g-1$, so once $\Sem(g,F)$ is computed (and then, counted) we do not need it anymore. Actually, since to compute $\Sem(F,g)$ we compute all $[S]$ for any $S\in \E(F,g)$, once one class is computed we can delete and only keep the number of elements of it.

Furthermore, if we are only interested on counting but not in enumerating the complete set of numerical semigroups in $\Sem(g)$ for some positive integer $g$, the code in Algorithm \ref{alg:1} can be modified conveniently to keep in memory the least possible Kunz-coordinates that are computed in the process.
\end{rem}

\end{document}